\documentclass[12pt,reqno]{amsart}

\usepackage{amsthm}
\usepackage{amscd}
\usepackage{amsfonts}
\usepackage{amssymb}
\usepackage{amsgen}
\usepackage{amsmath}
\usepackage{amsopn}

\theoremstyle{plain}
\newtheorem{thm}{Theorem}[section]
\newtheorem{lem}[thm]{Lemma}
\newtheorem{prop}[thm]{Proposition}
\newtheorem{cor}[thm]{Corollary}

\theoremstyle{definition}

 \newcommand{\nc}{\newcommand}

 \nc{\frakP}{{\mathfrak P}}
 \nc{\Z}{{\mathbb Z}}
 \nc{\R}{{\mathbb R}}
 \nc{\N}{{\mathbb N}}
 \nc{\ZN}{{{\mathbb N}_0}}
 \nc{\Q}{{\mathbb Q}}
 \nc{\CC}{{\mathbb C}}

 \nc{\calA}{{\mathcal A}}
 \nc{\calH}{{\mathcal H}}
 \nc{\calP}{{\mathcal P}}
 \nc{\gam}{{\gamma}}
 \nc{\gG}{{\Gamma}}
 \nc{\om}{{\omega}}
 \nc{\vep}{{\varepsilon}}
 \nc{\ga}{{\alpha}}
 \nc{\gl}{{\lambda}}
 \nc{\gb}{{\beta}}
 \nc{\gd}{{\delta}}
 \nc{\bfs}{{\bf s}}
 \nc{\gs}{{\sigma}}
 \nc{\gth}{{\theta}}
 \nc{\gS}{{\Sigma}}
 \nc{\gk}{{\kappa}}
  \nc{\gz}{{\zeta}}
 \nc{\tgz}{{\tilde{\zeta}}}
 \nc{\gO}{{\Omega}}
 \nc{\sif}{{\mathcal S}}
 \nc{\gt}{{\tau}}
 \nc{\Lra}{\Longrightarrow}
 \nc{\lra}{\longrightarrow}
 \nc{\lmaps}{\longmapsto}
 \nc{\fS}{{\mathfrak S}}
 \nc{\DD}{{\mathfrak D}}
 \nc{\Llra}{\Longleftrightarrow}
 \nc{\ol}{\overline}
 \nc{\ola}{\overleftarrow}
 \nc{\lms}{\longmapsto}
 \nc{\cv}{{{\mathsf c}{\mathsf v}}}
 \nc{\zq}{{\zeta_q}}
 \nc\qup{{q\uparrow 1}}
 \nc{\us}{\underset}
 \nc{\tn}{{\tilde{n}}}
 \nc{\gD}{{\Delta}}
 \nc{\bi}{{\bf i}}
 \nc{\bfone}{{\bf 1}}
\DeclareMathOperator{\sgn}{{sgn}}
\DeclareMathOperator{\len}{{len}}
\begin{document}
\title[Super Congruences Involving AMHS]{Some Families of Supercongruences Involving Alternating Multiple Harmonic Sums}

\author{Kevin Chen, Rachael Hong, Jerry Qu, David Wang, Jianqiang Zhao}
\address{Department of Mathematics, The Bishop's School, La Jolla, CA 92037}
\date{}

\email{chenk21@mit.edu}\email{rachelhong@college.harvard.edu}\email{jerry.qu.19@bishops.com}
\email{david.wang.18@bishops.com}\email{zhaoj@ihes.fr}

\subjclass[2010]{11A07, 11B68}

\keywords{Multiple harmonic sums, finite multiple zeta values, Bernoulli numbers, supercongruences}

\maketitle
\allowdisplaybreaks

\begin{abstract}
Let $p$ be a prime. In this short note we study some families
of supercongruences involving the following alternating sums
\begin{equation*}
\sum_{\substack{j_1+j_2+\cdots+j_n=2 p^r \\ p\nmid j_1 j_2 \cdots j_n }}
\frac{(-1)^{j_1+\cdots+j_b}}{j_1\cdots j_n}    \pmod{p^r},
\end{equation*}
which extend similar statements proved by Shen and Cai who treated the
cases when $n=4,5$. Our method works for arbitrary $n$.
\end{abstract}

\section{Introduction}
Over the past quarter of a century, multiple zeta values (MZVs) and their various
generalizations have been intensively studied by many mathematicians and physicists
due to their important applications in quite a few different areas of mathematics and
theoretical physics. These values are infinite series whose finite sums are
commonly called the multiple harmonic sums, defined as follows.
Let $\N$ and $\N_0$ be the set of positive integers and nonnegative integers, respectively.
For any $n,d\in\N$ and $\bfs=(s_1,\dots,s_d)\in\N^d$, we define
the \emph{multiple harmonic sums} (MHSs) by
\begin{equation*}
\calH_n(\bfs):= \sum_{n>k_1>\cdot>k_d>0} \frac{1}{k_1^{s_1}\dots k_d^{s_d}}.
\end{equation*}
For example, $\calH_{n+1}(1)$ is often called the $n$th harmonic number.

Very recently, a finite version of MZVs has emerged which has
been conjectured to be closely related to MZVs, see \cite[Ch.~8]{Zhao2015a}.
These values are essentially the MHSs truncated
at different primes and then taken residues modulo the corresponding primes.
Such congruences were first studied independently by the last author
in \cite{Zhao2008a,Zhao2011c} and Hoffman in \cite{Hoffman2015}. In general, it is
well-known that Bernoulli numbers play a very important role in these congruences, see
\cite{Sun2000} for some classical results. As an application,
in \cite{Zhao2007b} the last author proved, by using some special properties
of the double harmonic sums, that for every odd prime $p$
\begin{equation}\label{equ:BaseCongruence}
  \sum_{\substack{i+j+k=p\\ i,j,k\ge 1}} \frac1{ijk} \equiv -2 B_{p-3} \pmod{p},
\end{equation}
where $B_k$ are Bernoulli numbers defined by the generating series
$$\frac{t}{e^t-1}=\sum_{k=0}^\infty B_k\frac{t^k}{k!}.$$
Later, Ji gave an alternative simpler proof of \eqref{equ:BaseCongruence}
in \cite{Ji} using some combinatorial techniques.
Congruence \eqref{equ:BaseCongruence} has since been generalized by either increasing
the number of indices, changing the bound from $p$ to multiples of $p$ or $p$-powers,
and/or considering the corresponding supercongruences (see \cite{CaiShJi2015,MTWZ,Wang2014b,Wang2015,Zhao2014,ZhouCa2007}),
or even allowing the alternating
version of MHSs (see \cite{ShenCai2015b,ShenCai2012b}).

Our main results of this short note concern the following type of sums.
Let $\calP_p$ be the set of positive integers not divisible by $p$. For
$m,n,r,N\in\N$, we define
\begin{alignat*}{3}
Z_n(N,p):=&\,\sum_{\substack{l_1+l_2+\dots+l_n=N\\ l_1,\dots,l_n\in \calP_p }}
    \frac{1}{l_1 l_2\dots l_n} &&\quad \text{for } p|N,     \\
R_n^{(m)}(p^r):=&\,\sum_{\substack{l_1+l_2+\dots+l_n=mp^r\\ l_1,\dots,l_n\in \calP_p }}
    \frac{1}{l_1 l_2\dots l_n}=Z_n(mp^r,p) &&\quad \text{for } p\nmid m,\\
S_n^{(m)} (p^r):=&\, \sum_{\substack{l_1+l_2+\dots+l_n=mp^r\\ p^r>l_1,\dots,l_n\in \calP_p}}
    \frac{1}{l_1 l_2\dots l_n} &&\quad \text{for } p\nmid m.
\end{alignat*}
The primary goal of our study is to find nice and simple supercongruences
involving alternating sums defined as follows:
\begin{equation*}
\gs_n^{(b)}(N,p):= \sum_{\substack{l_1+l_2+\dots+l_n=N\\ l_1,\dots,l_n\in \calP_p }}
    \frac{(-1)^{l_1+\dots+l_b}}{l_1 l_2\dots l_n} \quad \text{for } p|N.
\end{equation*}
We will reduce these congruences to those of $Z_n(N,p)$ whose special cases $R_n^{(m)}(p^r)$
are closely related to $S_n^{(m)}(p^r)$ by Proposition~\ref{prop:RmS1}.
These results are motivated by the recent work of Shen and Cai \cite{ShenCai2015b} who
studied the above sums for $n=3,4$. In Theorem~\ref{thm:ASumRm} we generalize this to
arbitrary $n$ by using $R^{(m)}_n(p^r)$ with $m=1,2$.

\textbf{Acknowledgement.} We would like to thank the anonymous referee for the careful reading of
the manuscript and helpful comments and suggestions.

\section{Some useful lemmas}
We start with a formula expressing the sums $Z_n(N,p)$ in terms
of a modified version of multiple harmonic sums.

\begin{lem}\label{lem:Zn=MHS}
Let $n,N\in\N$ and $p$ be a prime. If $p|N$ then we have
\begin{equation}\label{equ:Z2MHS}
Z_n(N,p)=\frac{n!}N\sum_{\substack{1\le u_1<\dots<u_{n-1}<N \\
u_1,u_2-u_1,\dots,u_{n-1}-u_{n-2} ,u_{n-1}\in \calP_p }} \frac{1}{u_1\dots u_{n-1}}.
\end{equation}
\end{lem}
\begin{proof}
First, noting that $l_1+l_2+\dots+l_n=N$, we have
\begin{align*}
Z_n(N,p)=  &\, \frac{1}{N}\sum_{\substack{l_1+l_2+\dots+l_n=N\\ l_1,\dots,l_n\in \calP_p }}
    \frac{l_1+l_2+\dots+l_n}{l_1 l_2\dots l_n}     \\
= &\, \frac{n}{N}\sum_{\substack{l_1+\dots+l_{n-1}<N\\
l_1,\dots,l_{n-1},N-l_1-\cdots-l_{n-1}\in \calP_p}}
    \frac{1}{l_1 l_2\dots l_{n-1}}.
\end{align*}
Then one writes
\begin{equation*}
 \frac{1}{l_1 \dots l_{n-1}}=\frac{l_1 +\cdots+l_{n-1}}{l_1 \dots l_{n-1}(l_1 +\cdots+l_{n-1})}
\end{equation*}
to get
\begin{equation*}
 Z_n(N,p)=    \frac{n(n-1)}{N}\sum_{\substack{l_1+\dots+l_{n-2}<u_{n-1}<N\\
l_1,\dots,l_{n-2},u_{n-1}-l_1-\cdots-l_{n-2},N-u_{n-1}\in \calP_p}}
    \frac{1}{l_1 l_2\dots l_{n-2}u_{n-1}},
\end{equation*}
and continues in this way by using the substitutions $u_j=l_1+\dots+l_j$
for $1\le j<n$ to prove equation \eqref{equ:Z2MHS}. This completes the
proof of the lemma.
\end{proof}

\begin{lem} \label{lem:Srecurrence}
Suppose $m,n,r\in\N$ and $p$ is a prime with $p>n+1$. Then we have
\begin{equation*}
 S_n^{(m)}(p^{r+1}) \equiv (-1)^{m-1}\binom{n-2}{m-1} S_n^{(1)}(p^2) p^{r-1}   \pmod{p^{r+1}}.
\end{equation*}
\end{lem}
\begin{proof}
For all $n,a\in\N$, set
\begin{equation*}
\gam_n(a):=(-1)^{a+1}\frac{(a-1)!(n-1-a)!}{(n-1)!}.
\end{equation*}
By \cite[Lemma 2.3]{MTWZ}, we have
\begin{alignat*}{3}
S_n^{(m)}(p^{r+1})  &\, \equiv  p  \sum_{a=1}^{n-1} (-1)^{m-1}
            \binom{n-2}{m-1} \gam_n(a)S_n^{(a)}(p^r) & \pmod{p^{r+1}}\, \\
    &\, \equiv   (-1)^{m-1}\binom{n-2}{m-1} S_n^{(1)}(p^{r+1})   & \pmod{p^{r+1}}.
\end{alignat*}
So the lemma follows from \cite[(1.3)]{MTWZ} which says
\begin{equation*}
S_n^{(1)}(p^{r+1}) \equiv pS_n^{(1)} (p^r) \quad \pmod{p^{r+1}}
\end{equation*}
for all $r\ge 2$.
\end{proof}

\begin{prop}\label{prop:RmS1}
Let $m,n,r\in\N$ with $r\ge 2$. Then we have
\begin{alignat}{4}
R_n^{(m)}(p)\equiv&\, \sum_{a=1}^{n-1} \binom{m+n-a-1}{n-1} S_n^{(a)}(p)  & &\pmod{p}, \label{equ:Rnmr=1}\\
R_n^{(m)}(p^r)\equiv&\,   m \cdot S_n^{(1)}(p^2) p^{r-2} && \pmod{p^{r}}.\label{equ:Rnmr>1}
\end{alignat}
\end{prop}

\begin{proof}
Let $p$ be a prime number such that $p>n+1$.
For any $n$-tuples $(l_1,\cdots ,l_n)$ of integers in $\calP_{p}$ satisfying $l_1+\cdots +l_n=mp^r$,
we rewrite
\begin{equation*}
l_i=x_ip^r+y_i, \quad x_i \ge 0, \quad 1\le y_i<p^r, \quad y_i\in \calP_p, \quad 1 \le i \le n.
\end{equation*}
Since
\begin{equation*}
\Big(\sum_{i=1}^n x_i \Big)p^r+\sum_{i=1}^{n}{y_i}=mp^r,
\end{equation*}
there exists $1\le a<n$ such that
\begin{equation*}
\left\{ \begin{array}{ll}
 x_1+\cdots +x_n=m-a, \\
 y_1+\cdots +y_n=ap^{r}. \\
\end{array} \right.
\end{equation*}
For $1 \le a <n$, the equation $x_1+\cdots + x_n=m-a$ has $\binom {m+n-a-1}{n-1}$
nonnegative integer solutions. Hence, for all $r\ge 1$,
\begin{alignat*}{3}
R_n^{(m)}(p^r)\,& =\sum_{\substack{l_1+\cdots +l_n=mp^r  \\
 l_1,\cdots ,l_n\in \calP_p}}  \frac{1}{l_1l_2\cdots l_n}& \ \\
& =\sum_{a=1}^{n-1} \   \sum_{\substack{ x_1+\cdots +x_n=m-a\\
 y_1+\cdots +y_n=ap^r \\
 y_i\in \calP_p,y_i<p^r}}  \frac{1}{(x_1p^r+y_1)\cdots (x_np^r+y_n)} & \ \\
&\equiv \sum_{a=1}^{n-1} \binom{m+n-a-1}{n-1} S_n^{(a)}(p^r)  & & \pmod{p^r}  \\
&\equiv \sum_{a=1}^{n-1} \binom{m+n-a-1}{n-1}(-1)^{a-1}\binom{n-2}{a-1} S_n^{(1)}(p^2) p^{r-2} &&  \pmod{p^r},
\end{alignat*}
by Lemma \ref{lem:Srecurrence}. Note that the penultimate step holds for $r=1$ which
implies \eqref{equ:Rnmr=1}. However, the last step is valid only when $r\ge 2$.
So \eqref{equ:Rnmr>1} follows immediately from
\begin{align*}
&\, \sum_{a=1}^{n-1} (-1)^{a-1}\binom{n-2}{a-1} \binom{m+n-a-1}{n-1} \\
=&\, \sum_{a=1}^{n-1}  \binom{a+1-n}{a-1} \binom{m+n-a-1}{m-a} =m
\end{align*}
by the famous Chu--Vandermonde identity.
\end{proof}

\section{Alternating sums}
We now define the alternating version of the multiple harmonic sums. For convenience,
we denote by $\bar{s}$ a \emph{signed integer} for every $s\in\N$ and set $|\bar{s}|=s$
and $\sgn(\bar{s})=-1$. Let $s_j$ be either a positive integer
or a signed integer for all $j=1,\dots,d$. For any $n\in\N$, the alternating MHS is defined by
\begin{equation*}
\calH_n(s_1,\dots,s_d):= \sum_{n>k_1>\cdot>k_d>0} \frac{\sgn(s_1)^{k_1}\cdots \sgn(s_d)^{k_d}}
{k_1^{|s_1|}\dots k_d^{|s_d|}}.
\end{equation*}
For example, $\lim_{n\to \infty}\calH_n(\bar 1)$ is just the well-known
alternating harmonic series.

As variations of alternating MHSs, we have defined that
\begin{equation*}
\gs_n^{(b)}(N,p)= \sum_{\substack{l_1+l_2+\dots+l_n=N\\ l_1,\dots,l_n\in \calP_p }}
    \frac{(-1)^{l_1+\dots+l_b}}{l_1 l_2\dots l_n} \quad \text{for } p|N.
\end{equation*}
In this section, for each fixed $n\ge 4$, we will study some suitable linear combinations of
$\gs_n^{(b)}(N,p)$ for $b=1,\dots,n-1$. To this end, for any
$a\ge b\ge 0$, $d\ge 0$ and $\bfs=(s_1,\dots,s_d)\in \{1,\bar 1\}^d$, we define
\begin{align*}
F_a^{(b)} (\bfs, N,p):=&\, \sum_{\substack{N>i_1>\cdots>i_d>l_1+\dots+l_a,\ l_1,\dots,l_a\in \calP_p  \\
i_1,i_1-i_2,\dots,i_{d-1}-i_d, i_d-(l_1+\dots+l_a)\in \calP_p }} \hskip-2pt
    \frac{\sgn(s_1)^{i_1}\cdots \sgn(s_d)^{i_d}(-1)^{l_1+\dots+l_b}}{i_1 \cdots i_d l_1  \dots l_a}.
\end{align*}
Then it is easy to see that if $N$ is even then
\begin{align}\label{equ:gsDefn}
N\gs_n^{(b)}(N,p)=&\, (n-b)F_{n-1}^{(b)} (\emptyset, N,p)+bF_{n-1}^{(n-b)} (\emptyset, N,p),\\
F_a^{(b)}(\bfs,N,p)=&\, (a-b)F_{a-1}^{(b)} ((\bfs,1), N,p)+bF_{a-1}^{(a-b)} ((\bfs,\bar 1),N,p), \quad a\ge 1.
\label{equ:Frecursion}
\end{align}
Here, we have abused the notation by writing $(\bfs,1)=(s_1,\dots,s_d,1)$ and $(\bfs,\bar 1)=(s_1,\dots,s_d,\bar 1)$.
For $m\in\N_0$ and $n\in\N$, put
\begin{equation*}
    X_m:= (1_m),\quad Z_n:= (\{\bar 1\} , 1_{n-1}).
\end{equation*}
For $\bfs=(X_{w_1},Z_{w_2},\dots,Z_{w_l})$, we set $W_\bfs:=(w_1,w_2,\dots,w_l)$,
$\len(W_\bfs):=l$ and
\begin{equation}\label{equ:emptybfs}
    A_\emptyset:= 0,\ B_\emptyset:= 0,\ W_\emptyset:=(0),\
   P_\emptyset:= b.
\end{equation}
Otherwise, for $\bfs\ne \emptyset$, we define
\begin{align*}
   A_\bfs:= \sum_{2 \nmid i} w_i, \quad B_\bfs:= \sum_{2 \vert i} w_i, \quad
      P_\bfs:=  &
\left\{
  \begin{array}{ll}
    a-b-A_\bfs & \hbox{if $2\vert \len(W_\bfs)$;} \\
    b-B_\bfs & \hbox{if $2\nmid \len(W_\bfs)$.}
  \end{array}
\right.
\end{align*}
Finally, for all fixed $a\ge b\ge 0$ and $A,B\ge 0$, we put
\begin{align*}
   C_{a,b}(A, B)=C(A, B) := &
\left\{
  \begin{array}{ll}
    1 & \hbox{if $A,B=0$;} \\
    (b)_{B} & \hbox{if $A=0,B>0$;} \\
    (a-b)_{A} & \hbox{if $A>0,B=0$;} \\
    (a-b)_{A}(b)_{B}  & \hbox{if $A>0,B>0$,}
  \end{array}
\right.
\end{align*}
where $(x)_\ga=x(x-1)\cdots(x-\ga+1)$ is the Pochhammer symbol for the falling factorial.

\begin{lem} \label{lem:recursion}
Let $a,b\in \N_0$. Then for any fixed nonnegative integer $d\le a$,
\begin{equation*}
    F_a^{(b)}(\emptyset,N,p) = \sum_{\bfs\in\{1,\bar1\}^d} C_{a,b}(A_\bfs, B_\bfs)F_{a-A_\bfs-B_\bfs}^{(P_\bfs)}(\bfs,N,p).
\end{equation*}
\end{lem}
\begin{proof}
We will prove this by induction on $d$. If $d=0$ then there is only one term in the
sum corresponding to $\bfs=\emptyset$. Then the lemma holds by \eqref{equ:emptybfs}.
Now let $d\ge 1$ and suppose the lemma is true when $d$ is replaced by $d-1$. Observe that
any composition in $\{1,\bar1\}^d$ is produced by either $(\bfs,1)$ or $(\bfs,\bar1)$
for a unique $\bfs\in\{1,\bar1\}^{d-1}$. Further, it is easy to see that
\begin{align*}
   (A_{(\bfs,1)},B_{(\bfs,1)}) = &
\left\{
  \begin{array}{ll}
    (A_\bfs+1,B_\bfs) & \hbox{if $2 \nmid \len(W_\bfs)$;} \\
    (A_\bfs,B_\bfs+1) & \hbox{if $2 \vert \len(W_\bfs)$,}
  \end{array}
\right.\\
   (A_{(\bfs,\bar 1)},B_{(\bfs,\bar 1)}) = &
\left\{
  \begin{array}{ll}
    (A_\bfs,B_\bfs+1) & \hbox{if $2 \nmid \len(W_\bfs)$;} \\
    (A_\bfs+1,B_\bfs) & \hbox{if $2 \vert \len(W_\bfs)$.}
  \end{array}
\right.
\end{align*}
If $d<a$ and $2 \nmid \len(W_\bfs)$, then by \eqref{equ:Frecursion}
\begin{align*}
 &\,  C(A_\bfs, B_\bfs)F_{a-A_\bfs-B_\bfs}^{(b-B_\bfs)}(\bfs,N,p) \\
=&\,
C(A_\bfs, B_\bfs)\left[(a-b-A_\bfs)F_{a-A_\bfs-B_\bfs-1}^{(b-B_\bfs)}((\bfs,1),N,p)\right. \\
 &\,   +\left. (b-B_\bfs)F_{a-A_\bfs-B_\bfs-1}^{(a-b-A_\bfs)}((\bfs,\bar 1),N,p)\right] \\
=&\,
C(A_\bfs+1,B_\bfs)F_{a-(A_\bfs+1)-B_\bfs}^{(P_{(\bfs,1)})}((\bfs,1),N,P) \\
 &\,   +C(A_\bfs,B_\bfs+1)F_{a-A_\bfs-(B_\bfs+1)}^{(P_{(\bfs,\bar 1)})}((\bfs,\bar 1),N,p) \\
=&\,
C(A_{(\bfs,1)},B_{(\bfs,1)})F_{a-A_{(\bfs,1)}-B_{(\bfs,1)}}^{(P_{(\bfs,1)})}((\bfs,1),N,P) \\
  &\,  +C(A_{(\bfs,\bar 1)},B_{(\bfs,\bar 1)})F_{a-A_{(\bfs,\bar 1)}-B_{(\bfs,\bar 1)}}^{(P_{(\bfs,\bar 1)})}((\bfs,\bar 1),N,p).
\end{align*}
If $d<a$ and  $2 \vert \len(W_\bfs)$, then by \eqref{equ:Frecursion} again
\begin{align*}
  &\,  C(A_\bfs, B_\bfs)F_{a-A_\bfs-B_\bfs}^{(a-b-A_\bfs)}(\bfs,N,p) \\
=&\,
    C(A_\bfs, B_\bfs)\left[(b-B_\bfs)F_{a-A_\bfs-B_\bfs-1}^{(a-b-A_\bfs)}((\bfs,1),N,p)\right.\\
  &\,  +\left. (a-b-A_\bfs)F_{a-A_\bfs-B_\bfs-1}^{(b-B_\bfs)}((\bfs,\bar 1),N,p)\right] \\
=&\,C(A_\bfs,B_\bfs+1)F_{a-A_\bfs-(B_\bfs+1)}^{(P_{(\bfs,1)})}((\bfs,1),N,P)\\
  &\,  +C(A_\bfs+1,B_\bfs)F_{a-(A_\bfs+1)-B_\bfs}^{(P_{(\bfs,\bar 1)})}((\bfs,\bar 1),N,p)\\
=&\,C(A_{(\bfs,1)},B_{(\bfs,1)})F_{a-A_{(\bfs,1)}-B_{(\bfs,1)}}^{(P_{(\bfs,1)})}((\bfs,1),N,P)\\
  &\,  +C(A_{(\bfs,\bar 1)},B_{(\bfs,\bar 1)})F_{a-A_{(\bfs,\bar 1)}-B_{(\bfs,\bar 1)}}^{(P_{(\bfs,\bar 1)})}((\bfs,\bar 1),N,p).
\end{align*}
This finishes the induction proof of the lemma.
\end{proof}

\begin{cor}\label{cor:coeff}
Let $a,b\in\N$ with $a\ge b$. For all $\bfs\in\{1,\bar1\}^a$, we have
\begin{align*}
   C_{a,b}(A_\bfs,B_\bfs) = &
\left\{
  \begin{array}{ll}
    (a-b)!b! & \hbox{if $A_\bfs=a-b,B_\bfs=b$;} \\
    0 & \hbox{if $A_\bfs\ne a-b,B_\bfs\ne b$.}
  \end{array}
\right.
\end{align*}
\end{cor}
\begin{proof}
It is easy to see that $A_\bfs+B_\bfs=|W_\bfs|=|\bfs|=a$. If $C_{a,b}(A_\bfs,B_\bfs)\ne 0$, then by its definition
\begin{equation*}
    a-b-A_\bfs+1>0,     b-B_\bfs+1=-(a-b-A_\bfs)+1>0,
\end{equation*}
which imply that $A_\bfs = a-b, B_\bfs = b$ and $C_{a,b}(A_\bfs,B_\bfs) = (a-b)!b!$.
\end{proof}

\begin{cor} \label{cor:Fsum}
For all fixed $a\in\N$, we have
\begin{align*}
    \sum_{b=0}^a\binom{a}{b}F_a^{(b)}(\emptyset,2N,p)=\frac{N}{a+1}Z_{a+1}(N,p).
\end{align*}
\end{cor}
\begin{proof}
By Corollary \ref{cor:coeff},
$C(A_\bfs,B_\bfs)\ne 0$ for one and only one $b$ for every $\bfs\in\{1,\bar1\}^a$. Thus,
\begin{align*}
   &\,  \sum_{b=0}^a\binom{a}{b}F_a^{(b)}(\emptyset,2N,p)=\sum_{\bfs\in\{1,\bar1\}^a} a!F_0^{(0)}(\bfs,2N,p) \\
    =&\, a!\sum_{\substack{2N>i_1>\dots>i_a>0 \\ i_1,i_1-i_2,\dots,i_{a-1}-i_a,i_a\in \calP_p}}\frac{(1+(-1)^{i_1})\dots(1+(-1)^{i_a})}{i_1\dots i_a}
\end{align*}
As the term is nonzero only when all indices are even, we get
\begin{equation*}
   \sum_{b=0}^a\binom{a}{b}F_a^{(b)}(\emptyset,2N,p)= a!\sum_{\substack{N>i_1>\dots>i_a>0 \\ i_1,i_1-i_2,\dots,i_{a-1}-i_a,i_a\in  \calP_p}}\frac{1}{i_1\dots i_a}.
\end{equation*}
We can now finish the proof of the corollary by applying Lemma \ref{lem:Zn=MHS}.
\end{proof}

\begin{thm}\label{thm:ASumRm}
Let $n,N$ be two positive integers and $p$ a prime. If $p|N$ then we have
\begin{equation*}
\sum_{b=1}^{\lfloor n/2 \rfloor}  \ga_{n,b} \binom{n}{b}
\gs_n^{(b)}(2N,p) = \frac12Z_n(N,p)-Z_n(2N,p),
\end{equation*}
where $\ga_{n,b}=1$ except for $\ga_{n,n/2}=1/2$ when $n$ is even. In particular, for
every $r\in\N$ and prime $p$ we have
\begin{align*}
\sum_{b=1}^{\lfloor n/2 \rfloor}  \ga_{n,b} \binom{n}{b}
\gs_n^{(b)}(2p^r,p) = &\, \frac12 R^{(1)}_n(p^r)-R^{(2)}_n(p^r) \\
\equiv&\,  -\frac32 S_n^{(1)}(p^2) p^{r-1} \pmod{p^r}.
\end{align*}
\end{thm}
\begin{proof}
For even $N$, we have $\gs_n^{(b)}(N,p) = \gs_n^{(n-b)}(N,p)$ and therefore we get
\begin{align*}
   &\, 2 N\sum_{b=0}^{\lfloor n/2 \rfloor}  \ga_{n,b} \binom{n}{b} \gs_n^{(b)}(2N,p)
     =   \frac{1}2  \sum_{b=0}^n \binom{n}{b} 2N\gs_n^{(b)}(2N,p) \\
     =  &\,    \frac{1}2 \sum_{b=0}^n \binom{n}{b} (n-b)F_{n-1}^{(b)}(\emptyset,2N,p)
     + \frac{1}2 \sum_{b=0}^n \binom{n}{b} bF_{n-1}^{(n-b)}(\emptyset,2N,p)
 \end{align*}
by \eqref{equ:gsDefn}. Using substitution $b\to n-b$ in the second sum, we get
\begin{align*}
2 N\sum_{b=0}^{\lfloor n/2 \rfloor}  \ga_{n,b} \binom{n}{b} \gs_n^{(b)}(2N,p)
   = &\, \sum_{b=0}^n (n-b) \binom{n}{b} F_{n-1}^{(b)}(\emptyset,2N,p)\\
     =&\, n \sum_{b=0}^{n-1} \binom{n-1}{b} F_{n-1}^{(b)}(\emptyset,2N,p)
     =N Z_n(N,p),
\end{align*}
by Corollary \ref{cor:Fsum} with $a=n-1$. Therefore,
\begin{align*}
\sum_{b=1}^{\lfloor n/2 \rfloor}  \ga_{n,b} \binom{n}{b} \gs_n^{(b)}(2N,p)
=&\, \sum_{b=0}^{\lfloor n/2 \rfloor}  \ga_{n,b} \binom{n}{b}
\gs_n^{(b)}(2N,p) - \gs_n^{(0)}(2N,p)\\
=&\, \frac12Z_n(N,p)-Z_n(2N,p)
\end{align*}
since  $\gs_n^{(0)}(2N,p) = Z_n(2N,p)$. The final congruence of the theorem follows easily
from Proposition~\ref{prop:RmS1}. This completes the proof of the theorem.
\end{proof}

\begin{cor}\label{cor:ASumRm}
Let $n\in\N$ and $p$ be a prime such that $p>n+1$. Then we have
\begin{equation*}
\sum_{b=1}^{\lfloor n/2 \rfloor}  \ga_{n,b} \binom{n}{b}\gs_n^{(b)}(2p,p)
\equiv
\left\{
\aligned
 & \frac{n!}2 B_{p-n} \pmod{p}, && \text{if $2\nmid n$};\\
 & -\frac{n!}{2}\sum_{\substack{a+b=n\\ a,b\ge 3}}\frac{B_{p-a}B_{p-b}}{ab}  \pmod{p}, && \text{if $2|n$}.
\endaligned
\right.
\end{equation*}
\end{cor}
\begin{proof}
This follows easily from Theorem \ref{thm:ASumRm},
\cite[Main Theorem]{ZhouCa2007},  \cite[Lemma 3.5 and Corollary 3.6]{MTWZ} (for $n$ odd) and
\cite[Theorem~1 and Corollary\ 1]{Wang2015} (for $n$ even).
\end{proof}

\begin{cor}\label{cor:n=4}
Let $r \in\N$ and $p>4$ be a prime. We have
\begin{alignat}
4\gs_4^{(1)}(2p^r,p)+3\gs_4^{(2)}(2p^r,p) \equiv&\, 0  &&\pmod{p^r}, \label{equ:gs4}\\
\gs_5^{(1)}(2p^r,p)+2\gs_5^{(2)}(2p^r,p) \equiv&\,
    6 B_{p-5} p^{r-1} &&\pmod{p^r}. \label{equ:gs5}
\end{alignat}
\end{cor}
\begin{proof}
It follows from \cite[Theorem~1]{Wang2015}, \cite[Theorem~1.1]{Zhao2014},
\cite[Main Theorem]{ZhouCa2007}, and \cite[Theorem~2]{Wang2014b}  that
\begin{equation*}
   S_4^{(1)}(p^2)\equiv 0, \quad S_5^{(1)}(p^2)\equiv -20 B_{p-5} p \pmod{p^2}.
\end{equation*}
So Theorem \ref{thm:ASumRm} yields the corollary immediately.
\end{proof}

In fact, this note was motivated by Shen and Cai's proof of \eqref{equ:gs4}
and a finer version of \eqref{equ:gs5} in \cite{ShenCai2012b}.
Now it follows from \cite[Theorem~4]{Wang2015} and \cite[Theorem~1.1]{MTWZ}  that
\begin{equation*}
S_6^{(1)}(p^2)\equiv -\frac{20}3 B_{p-3}^2 p,\quad
S_7^{(1)}(p^2)\equiv -504 B_{p-7} p \pmod{p^2},
\end{equation*}
and, by similar computation (see \cite{ChenZ2017} for details)
\begin{alignat*}{3}
S_8^{(1)}(p^2)\equiv&\,  -\frac{1792}5 B_{p-3}B_{p-5} p, \quad
& S_9^{(1)}(p^2)\equiv &\, -\frac{32}3 (2283 B_{p-9}+7 B_{p-3}^3) p \pmod{p^2}.
\end{alignat*}
Therefore, by Theorem \ref{thm:ASumRm}, modulo $p^r$ ($r\ge 2$), we have
\begin{align*}
  6\gs_6^{(1)}(2p^r,p)+15\gs_6^{(2)}(2p^r,p)+10\gs_6^{(3)}(2p^r,p)
\equiv 10B_{p-3}^2  p^{r-1}, \\
 7\gs_7^{(1)}(2p^r,p)+21\gs_7^{(2)}(2p^r,p)+35\gs_7^{(3)}(2p^r,p) \equiv
 756 B_{p-7}  p^{r-1},  \\
 8\gs_8^{(1)}(2p^r,p)+28\gs_8^{(2)}(2p^r,p)+56\gs_8^{(3)}(2p^r,p)+35\gs_8^{(4)}(2p^r,p) \\
   \equiv \frac{2688}{5} B_{p-3}B_{p-5}\,   p^{r-1} , \\
  9\gs_9^{(1)}(2p^r,p)+36\gs_9^{(2)}(2p^r,p)
+84\gs_9^{(3)}(2p^r,p)+126\gs_9^{(4)}(2p^r,p) \\
 \equiv
  16\Big(2283 B_{p-9}+7 B_{p-3}^3\Big)\, p^{r-1}.
\end{align*}
By combining Theorem \ref{thm:ASumRm} and the numerical results of $S_n^{(1)}(p^2)$
obtained in \cite{ChenZ2017}, one can derive easily
similar explicit formulas for all $n\le 12$.


\begin{thebibliography}{10}

\bibitem{CaiShJi2015}
T. Cai, Z. Shen and L. Jia, A congruence involving harmonic sums modulo $p^\ga q^\gb$,
\emph{Intl.\ J.\ Number Theory}  \textbf{13}  (2017), pp.\ 1083--1094.

\bibitem{ChenZ2017}
K.\ Chen and J.\ Zhao, Supercongruences involving multiple harmonic sums and Bernoulli numbers,
\emph{J.\ Integer Sequences} \textbf{20} (2017), Article 17.6.8.

\bibitem{Hoffman2015}
M.E.\ Hoffman,Quasi-symmetric functions and mod $p$ multiple harmonic sums,
\emph{Kyushu J.\ Math.} \textbf{69} (2015), pp.\ 345--366.

\bibitem{Ji}
C.\ Ji, A simple proof of a curious congruence by Zhao,
\textit{Proc.\ Amer.\ Math.\ Soc.} \textbf{133} (2005), pp.\ 3469--3472.

\bibitem{MTWZ}
M.\ McCoy, K.\ Thielen, L.\ Wang and J. Zhao,
A family of super congruences involving multiple harmonic sums.
\emph{Intl.\ J.\ Number Theory} \textbf{13} (2017), pp.\ 109--128.

\bibitem{ShenCai2015b}
T. Cai and Z. Shen, Super congruences involving alternating harmonic sums modulo prime powers, arxiv: 1503.03156.

\bibitem{ShenCai2012b}
Z.\ Shen and T. Cai, Congruences for alternating triple harmonic sums,
\emph{Acta Math. Sinica (Chin. Ser.)}, \textbf{55} (2012), pp.\ 737--748.

\bibitem{Sun2000}
Z.-W.\ Sun, Congruences concerning Bernoulli numbers and Bernoulli polynomial,
\emph{Disc.\ Applied Math.} \textbf{105} (2000), pp.\ 193--223.

\bibitem{Wang2014b}
L.\ Wang, A new curious congruence involving multiple harmonic sums, \emph{J.\ Number Theory} \textbf{154} (2015), pp.\ 16--31.

\bibitem{Wang2015}
L.\ Wang, New congruences on multiple harmonic
sums and Bernoulli numbers. arXiv:1504.03227.

\bibitem{Zhao2007b}
J.\ Zhao, Bernoulli numbers, Wolstenholme's Theorem, and $p^5$ variations of Lucas' Theorem,
\emph{J.\ Number Theory} \textbf{123} (2007), pp.\ 18--26.

\bibitem{Zhao2008a}
J.\ Zhao, Wolstenholme type theorem for multiple harmonic
sums, \emph{Int.\ J.\ Number Theory} \textbf{4} (2008), pp.\ 73--106.

\bibitem{Zhao2011c}
J.\ Zhao, Mod $p$ structure of alternating and non-alternating multiple
harmonic sums.
\emph{J.\ Th\'eor.\ Nombres Bordeaux} \textbf{23} (2011), pp.\ 259--268. (\textbf{MR} 2780631)

\bibitem{Zhao2014}
J.\ Zhao, Congruences involving multiple harmonic sums and finite multiple zeta values.
\emph{Analysis, Geometry and Number Theory} (2) 2017, pp.\ 59--75.
doi: 10.19272/201712501003. 

\bibitem{Zhao2015a}
J.\ Zhao,
Multiple Zeta Functions, Multiple Polylogarithms and Their Special Values, Series on
Number Theory and Its Applications, vol.\ \textbf{12},
World Scientific Publishing Co. Pte. Ltd., Hackensack, NJ, 2016. 

\bibitem{ZhouCa2007}
X.\ Zhou and T.\ Cai, A generalization of a curious congruence on harmonic sums,
\emph{Proc.\ Amer.\ Math.\ Soc.} \textbf{135} (2007), pp.\ 1329--1333.

\end{thebibliography}
\end{document}